\documentclass[a4paper,12pt]{article}

\usepackage{graphicx}
\usepackage{amsmath, amsfonts, amsthm, amssymb, enumerate, hyperref,
fancyhdr, fullpage}
\usepackage[all]{xy}

\newcommand{\ignore}[1]{}

\newtheorem{thm}{Theorem}
\newtheorem{lemma}{Lemma}
\newtheorem{prop}{Proposition}
\newtheorem{conj}{Conjecture}

\title{The Arc-Weighted Version of the Second Neighborhood Conjecture}

\author{Tyler Seacrest\thanks{The Universty of Montana Western,
710 S Atlantic St,
Dillon, MT 59725,
United States} \\
\texttt{t\_seacrest@umwestern.edu}}

\date{\today}

\begin{document}

\maketitle

\abstract{Seymour conjectured that every oriented simple graph contains a vertex whose second neighborhood is at least as large as its first.  Seymour's conjecture has been verified in several special cases, most notably for tournaments by Fisher~\cite{Fisher96}.

One extension of the conjecture that has been used by several researchers is to consider vertex-weighted digraphs.  In this paper we introduce a version of the conjecture for arc-weighted digraphs.   We prove the conjecture in the special case of arc-weighted tournaments, strengthening Fisher's theorem.  Our proof does not rely on Fisher's result, and thus can be seen as an alternate proof of said theorem.

}

\section{Introduction}

Unless otherwise noted, all digraphs in this paper are oriented simple graphs, and thus do not contain loops or two-cycles.    We wil use $V(D)$ to denote the set of vertices of a digraph $D$, $A(D)$ to denote the set of arcs or edges.

Given a digraph $D$ and vertices $u$ and $v$, we call $u$ an \emph{$n$th out-neighbor} of $v$ if the shortest directed path connecting $v$ to $u$ has $n$ edges.  Let $N_n^+(v)$ be the set of all $n$th out-neighbors of $v$.   We will focus on $N_1^+(v)$ and $N_2^+(v)$, and we note that these are disjoint.  We will use $N_1^-(v)$ and $N_2^-(v)$ to refer to the sets of first and second in-neighbors, defined analogously to out-neighbors.  If not specified, the term \emph{neighbors} refers to first out-neighbors. 

If $|N_1^+(v)| \leq |N_2^+(v)|$, we will call $v$ a \emph{weakly expanding vertex} or a \emph{Seymour vertex}.  If $v$ is not Seymour vertex, then $|N_1^+(v)| > |N_2^+(v)|$ and we say $v$ is \emph{strongly contracting}.

Seymour made the following conjecture, which has become known as Seymour's Second Neighborhood Conjecture.

\begin{conj}[Seymour, see~\cite{DeanLatka95}]
\label{conj:original}
Every digraph without loops or two-cycles  contains a Seymour vertex.
\end{conj}

Early work on this conjecture focused on tournaments, which was known as Dean's Conjecture.  This important special case was proven by Fisher~\cite{Fisher96} using the losing density of a digraph.  A weight function $\ell: V \to [0, 1]$ is a \emph{losing density} if the weights sum to one and every vertex has more weight on its out-neighbors than its in-neighbors.  That is, $\sum_{v \in V} \ell(v) = 1$ and for every vertex $u$, $\sum_{v \to u} \ell(v) \leq \sum_{u \to v} \ell(v)$.   He first proved that every digraph (not just tournaments) contained a losing density using Farka's lemma (see~\cite{BertsimasTsitsiklis97}) from linear programming.  He then used the existence of a losing density on a tournament to show that every tournament has a Seymour vertex.

Later, Havet and Thomass\'{e}~\cite{HavetThomasse00} gave a short and elegant proof of Fisher's Theorem using median orders.  A \emph{median order} is an ordering of the vertices $v_1, \ldots, v_n$ of a digraph that minimizes the number of pairs $v_i \to v_j$ where $j < i$.   They proved that for a tournament, $v_n$ in a median order is a Seymour vertex.  They went on to show the existence of two Seymour vertices in the case the tournament had no sink.

Several other partial results are known.  Kaneko and Locke~\cite{KanekoLocke01} proved Seymour's conjecture for graphs with minimum out-degree six or less.    Chen, Shen, and Yuster~\cite{ChenShenYuster03} show every oriented graph contains a vertex $v$ such that $|N_1^+(v)| > \gamma |N_2^+(v)|$ for $\gamma  \approx 0.657298$ is the real root of the equation $2 x^3 + x^2 - 1 = 0$.  Fidler and Yuster~\cite{FidlerYuster07} introduced the vertex-weighted version of the conjecture, and proved, among other results, that the conjecture is true for orientations of a complete graph minus a matching, which generalized Fisher's Theorem.    Ghazal~\cite{Ghazal12}, also used vertex-weighted digraphs, and also extended Fisher's theorem.  In this case, the author showed that the conjecture holds for orientations of generalized stars or threshold graphs.

We examine the second neighborhood conjecture for arc-weighted digraphs, defined in Section~\ref{sec:arc-weighted}.  The arc-weighted version is equivalent to the original conjecture, as shown in Proposition~\ref{prop:equiv}.  For tournaments, however, the arc-weighted version and original version are not equivalent, because the auxiliary graph used in the proof of Proposition~\ref{prop:equiv} is not a tournament.  Our main result, given in Section~\ref{sec:tournaments}, shows that the arc-weighted tournaments do in fact satisfy the second neighborhood conjecture.  In fact, we do not explicitly use Fisher's theorem in our proof, and thus our proof can be seen as an alternate proof of Fisher's Theorem, albeit one that is not as radically different as Havet and Thomass\'{e}'s proof.  Along the way, in Section~\ref{sec:vertex-weighted}, we will develop some results related to vertex-weighted digraphs.

\section{Vertex-Weighted Digraphs}
\label{sec:vertex-weighted}

Seymour's second neighborhood conjecture can be expanded to vertex-weighted digraphs, which are digraphs that have a weight function $\eta$ which assigns each vertex a nonnegative real number.  This extends naturally to a weight function on sets of vertices $S$, where $\eta(S) = \sum_{v \in S} \eta(v)$.   We define the \emph{$n$th neighborhood weight} of a vertex $v$ to be $\eta(N_n^+(v))$.  In a weighted digraph, $v$ is \emph{weakly expanding} if $\eta(N_1^+(v)) \leq \eta(N_2^+(v))$, and is \emph{strongly contracting} if $\eta(N_1^+(v)) > \eta(N_2^+(v))$.  A weakly expanding vertex is called a \emph{Seymour vertex}.

Seymour's conjecture then becomes
\begin{conj}[Seymour, vertex-weighted version]
Every weighted digraph without loops or two-cycles contains a Seymour vertex.
\end{conj}
Note that while this may first appear to be a stronger conjecture, it is implied by the original version.

Like Fisher, we will use Farkas' Lemma.
\begin{lemma}[Farkas' Lemma (see~\cite{BertsimasTsitsiklis97})]
\label{lemma:farkas}
For any matrix $A$ and vector $\mathbf{b}$, exacly one of the following holds.
\begin{enumerate}
	\item $A \mathbf{x} = \mathbf{b}$ for some $\mathbf{x}$ such that $\mathbf{x} \geq 0$.
	\item $\mathbf{p}^T A \geq 0$ for some $\mathbf{p}$ such that $\mathbf{p}^T \mathbf{b} < 0$.
\end{enumerate}
\end{lemma}

Let the \emph{reverse} of $D$, denoted $\overleftarrow{D}$, be the digraph with vertex set $V(D)$ and edge set $\{ vu \mid uv \in A(D)\}$.   Using Farkas' lemma we show the following.

\begin{thm}
\label{thm:expanding-or-contracting}
Given any digraph $D$ and its reverse $\overleftarrow{D}$, there is either a vertex-weighting of $D$ where each vertex is weakly expanding or a vertex-weighting of $\overleftarrow{D}$ where each vertex is strongly contracting, and not both. 
\end{thm}

\begin{proof}
Let $D$ be a digraph on vertex set $v_1 \ldots v_n$.  Let $N$ be a matrix with entry $n_{ij}$ in the $i$th row and $j$th column, where $n_{ij}$ is $-1$ if $v_j$ is a first out-neighbor of $v_i$, and $n_{ij} = 1$ if $v_j$ is a second out-neighbor of $v_i$.  Let $A$ be the matrix
$$
A = \begin{bmatrix} N & -I \\ 1 \cdots 1 & 0 \cdots 0\end{bmatrix}.
$$
Let $\mathbf{x}^T = (x_1, \ldots, x_n, s_1, \ldots, s_n)$ and let $\mathbf{b}^T = (0, \ldots, 0, 1)$.  Then $A \mathbf x = \mathbf b$ for $\mathbf x \geq 0$ has a solution if and only if there are non-trivial vertex-weights $x_1, \ldots, x_n$ such that every vertex has as much weight in its second neighborhood then its first.  In other words, every vertex is weakly expanding.   Let $\mathbf{p}^T = (p_1, \ldots, p_{n+1})$.  Then the system $\mathbf{p}^T A \geq 0$, $\mathbf p^T \mathbf b < 0$ means that there are nonpositive weights $p_1, \ldots, p_n$ of the vertices of $D$ such that $\eta(N_1^-(v)) < \eta(N_2^-(v))$ for all $v$.  If we use the nonnegative weights $-p_1, \ldots, -p_n$ instead, we have $\eta(N_1^-(v)) > \eta(N_2^-(v))$.  In other words, there is a weighting of the vertices such that every vertex in $\overleftarrow{D}$ is strongly contracting.   Farkas' lemma said that exactly one of these must hold for every digraph, which gives the result.
\end{proof}

The second neighborhood conjecture says it is impossible for every vertex of $\overleftarrow{D}$ to be strongly contracting, and hence we have that the following is equivalent to the second neighborhood conjecture.
\begin{conj}
Every digraph without loops or two-cycles has a vertex-weighting such that every vertex is weakly expanding.
\end{conj}

Theorem~\ref{thm:expanding-or-contracting}, while not explicitly stated or proved previously, was strongly inspired by Fisher~\cite{Fisher96}, especially his use of the Farkas' lemma.  Also, Fisher's main result can be restated as
\begin{thm}[Fisher~\cite{Fisher96}]
\label{thm:expanding}
Every tournament $D$ has a vertex-weighting such that every vertex in $\overleftarrow{D}$ is weakly expanding. (In particular, this is achieved by a losing density.)
\end{thm}
Fisher then concluded by a simple calculation that $D$ has a Seymour vertex.  We could also go from Theorem~\ref{thm:expanding} to the existence of a Seymour vertex immediately by applying Theorem~\ref{thm:expanding-or-contracting}.   

Another consequnce of Theorem~\ref{thm:expanding-or-contracting} is that a minimum counterexample to the vertex-weighted second neighborhood conjecture must have a dual counterexample.   That is, consider a minimum digraph $D$ such that $D$ has a vertex-weight function $\eta$  so that it is a counterexample to the second neighborhood conjecture.  There must also be a weight function $\eta'$ so that $\overleftarrow{D}$ is a counterexample:  otherwise, one could find an expansion weighting of $D$ by Theorem~\ref{thm:expanding-or-contracting} and subtract it from the counterexample weighting until a vertex reaches weight zero.  Once there is a vertex of weight zero, deleting that vertex yields a smaller counterexample.

We will end this section with a result that we will need later on, showing the second neighborhood conjecture is true for vertex-weighted digraphs where if every arc is contained in a directed triangle.  This result is along the same lines as one of the results by Brantner, Brockman, Kay, and Snively~\cite{BrantnerBrockmanKaySnively09}, who showed the second neighborhood conjecture holds for digraphs without a transitive triangle.

\begin{prop}
\label{prop:triangle-full}
Let $D$ be a vertex-weighted digraph such that every edge is contained in a directed triangle. Then $D$ satisfies the second neighborhood conjecture.
\end{prop}

\begin{proof}
By way of contradiction, suppose every edge of $D$ is contained in a directed triangle, and yet $D$ does not satisfy the second neighborhood conjecture.  For any $v \in V(D)$, $N_1^-(v) \subseteq N_2^{+}(v)$.  Thus we have the string of inequalitites
$$
\sum_{v \in V(D)} \eta(N_1^-(v)) \leq \sum_{v \in V(D)} \eta(N_2^{+}(v)) < \sum_{v \in V(D)} \eta(N_1^+(v)) = \sum_{v \in V(D)} \eta(N_1^-(v)),
$$
which yields a contradiction.
\end{proof}

\section{Arc-Weighted Digraphs}
\label{sec:arc-weighted}

Seymour's second neighborhood conjecture can also be expanded to arc-weighted digraphs, which are digraphs with a weight function $w$ assigning each arc and nonnegative real number.  Given a vertex $v$, the \emph{first neighborhood weight} of $v$ is denoted $\alpha^D_v$ and is defined as $\sum_{u \in N_1^+(v)} w(vu)$.   

We will now define the second neighborhood weight of $v$, which will be denoted $\beta^D_v$.  Let $s$ be a vertex at the end of a path of length 2 starting at $v$.   Define $\beta_v(s)$ to be the maximum over all $u$ such that $v \to u \to s$ of $w(us) - w(vs)$, and $\beta_v(s)$ is $0$ if this maximum is negative.  Notationally,
$$
\beta_v(s) = \max \left( \{0\} \cup \bigcup_{u:  v \to u \to s} \{w(us) - w(vs)\} \right).
$$
For the purpose of this definition, if $vs$ is not an edge, take $w(vs) = 0$.   The \emph{second neighborhood weight} $\beta^D_v$ is given by $\sum_{s} \beta_v(s)$.  Finally, the \emph{neighborhood weight difference} of a vertex $v$, denoted by $\delta^D_v$, is $\beta^D_v - \alpha^D_v$.  When $D$ is clear from context, we will use $\alpha_v$, $\beta_v$, and $\delta_v$ instead of $\alpha^D_v$, $\beta^D_v$, and $\delta^D_v$.  For an arc-weighted digraph, a vertex is \emph{weakly expanding} if $\delta_v \geq 0$.  Similarly, a vertex is  \emph{strongly contracting} if $\delta_v < 0$.  If every vertex in the graph is weakly expanding, the graph is weakly expanding, and the same is true for strongly contracting.

Consider the following example.

\begin{center}
\includegraphics[scale=0.8]{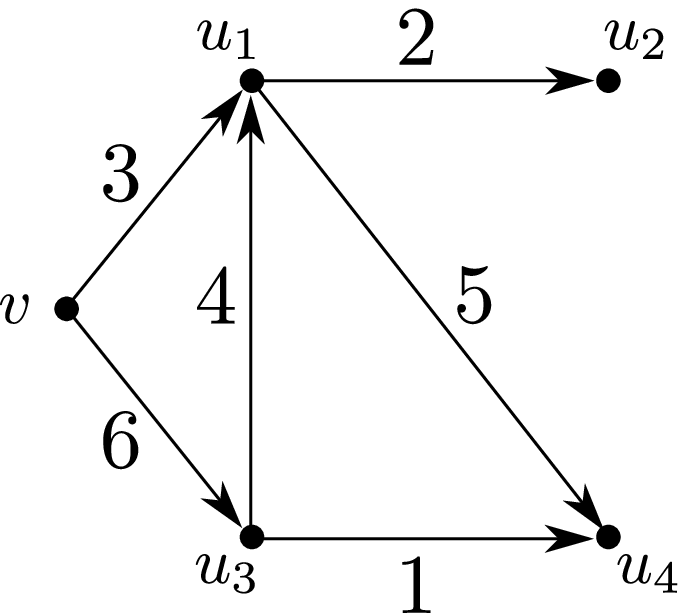}
\end{center}
Here, the first neighborhood weight of $v$ is the sum of the out-going arcs, and hence $\alpha_v = 3 + 6 = 9$.   To compute the second neighborhood weight, we need the value $\beta_v(s)$ of every neighbor $s$ at the end of a path of length $2$.  Note that the possibilities for $s$ are $u_1$, $u_2$, and $u_4$.    We see $\beta_v(u_2) = 2$ because there is only one way to reach $u_2$, and it ends on an arc of weight $2$.  There are two ways to get to $u_4$, one that ends on an arc of length $5$, and another that ends on an arc of weight $1$.  Taking the maximum of these two, we get $\beta_v(u_4) = 5$.  Finally, there is only one way to get to $u_1$, and it ends on an arc of weight $4$.  But since $v u_1$ is an arc of weight $3$, we subtract $3$ from $4$ to get $\beta_v(u_1) = 1$.  The total second neighborhood weight is then $\beta_v = \beta_v(u_1) + \beta_v(u_3) + \beta_v(u_4) = 1 + 5 + 2 = 8$.  Thus we have $\delta_v = \beta_v - \alpha_v = 8 - 9 = -1$, which shows that $v$ is strongly contracting.

Notice that an arc of weight zero is not the same as a missing arc.  If there is an arc of weight zero from $u$ to $v$, $u$ will still count all of $v$'s neighbors as potential second neighbors, where as a missing arc from $u$ to $v$ means the neighbors of $v$ do not count as potential second neighbors of $u$.

The arc-weighted second neighborhood conjecture is as follows.

\begin{conj}
\label{conj:arc-weighted}
Every arc-weighted digraph $D$ without loops or two-cycles contains a vertex $v$ such that $\delta_v \geq 0$.  
\end{conj}

The following proposition shows the equivalence to the original conjecture via an auxiliary graph construction.

\begin{prop}
\label{prop:equiv}
The arc-weighted version of the second neighborhood conjecture (Conjecture~\ref{conj:arc-weighted})  is equivalent to the orignal conjecture (Conjecture~\ref{conj:original}).
\end{prop}

\begin{proof}
If we assume the arc-weighted version, then the original follows by applying the arc-weighted version to a digraph where each arc has weight 1.

Now suppose there is a counterexample $D$ to the arc-weighted version of the conjecture.  Thus every vertex is strongly contracting, which means $\delta_v < 0$ for every vertex $v$. We can assume $D$ has no arcs of weight zero, as removing these  cannot make $\delta_v$ larger.   By the fact that the rationals are dense in the reals, we can assume the arc-weights on $D$ are rational.  By scaling the weights with a large enough multiple, we can then assume the arc-weights on $D$ are positive integers. 

We now create an auxiliary digraph $D'$ without arc-weights that will be a counterexample to the original conjecture.  To form $D'$ from $D$, replace each vertex $v$ of $D$ with a set of vertices $S_v$ such that $|S_v| = \max_{u \in N_1^-(v)} w(uv)$, and arbitrarily order the vertices of each $S_v$.  If there is an edge from $u$ to $v$ in $D$, place arcs between every vertex of $S_u$ to the first $w(uv)$ vertices of $S_v$ in $D'$.  

We now show that $D'$ is a counterexample to the original conjecture. Given any vertex $x \in S_v$, $x$ has as out neighbors the first $w(vu)$ vertices in $S_u$ for all $u \in N_1^+(v)$.  Thus we see that $|N_1^+(x)| = \sum_{u \in N_1^+(v)} w(vu)$, which is the first neighborhood weight of $v$, and hence $|N_1^+(x)| = \alpha_v$. 

$N_2^+(x)$ consists of vertices in $S_s$ such that there exists a $u$ with $v \to u \to s$ in $D$.  In particular, $S_s \cap N_2^+(x)$ consists of the first $w(us)$ elements of $S_s$, where $u$ is chosen such that $v \to u \to s$ and $w(us)$ is maximized.    However, there may be some elements of $S_s$ that are actually first neighbors, so the number of neighbors is reduced by $w(vs)$.  Putting this all together, we get  
$$
|N_2^+(x)| = \sum_{s \in N_1^+(v) \cup  N_2^+(v)} \max_{v \to u \to s} \{w(us) - w(vs), 0\}.
$$
This is the second neighborhood weight of $v$, and hence $|N_2^+(x)| = \beta_v$.  
  
Since $D$ is a counterexample to the arc-weighted conjecture, we know $\alpha_v > \beta_v$.  Hence $|N_1^+(x)| > |N_2^+(x)|$ for all $x \in V(D')$, which shows that $D'$ is a counterexample to the original conjecture.
\end{proof}

Notice that if $D$ has vertex-weights $\eta$ and is a counterexample to the second neighborhood conjecture, define an arcweight function $w$ where $w(uv) = \eta(v)$.  Then $D$ with arc-weights $w$ is a counterexample to the arc-weighted version of the conjecture.

\section{Arc-Weighted Tournaments}
\label{sec:tournaments}

While the arc-weighted version of the conjecture follows from the original conjecture, it is not true that the arc-weighted version for tournaments is a simple consequence of Fisher's Theorem.  This is because if an arc-weighted tournament $D$ were to undergo the the transformation to $D'$ from the proof of Proposition~\ref{prop:equiv}, $D'$ would likely not be a tournament.

Nor do the proofs of Fisher or Havet and Thomass\'{e} easily extend to the case of arc-weighted tournaments.  

To extend Fisher's proof, one would need to extend the idea of losing density to an arc-weighted tournament. One natural candidate is as follows:  given a digraph $D$ with arc-weights given by $w$, we say $\ell$ is a \emph{losing density} if, for all $v$, 
$$
\sum_{xv \in A(D)} w(xv) \ell(x) \leq \sum_{vy \in A(D)} w(vy) \ell(y).
$$ 
 Fisher's main result for tournaments without arc-weights is that the losing density, as a vertex-weighting of $\overleftarrow{D}$, is weakly expanding at every vertex.  However, a losing density is not necessarily weakly expanding at every vertex of $\overleftarrow{D}$ once arc-weights are introduced.  In the example below (adapted from~\cite{Fisher96}), $D$ is shown with an arc-weighted losing density, but the vertex in the upper left corner is not weakly expanding in $\overleftarrow{D}$.  

\begin{center}
\includegraphics[scale=0.8]{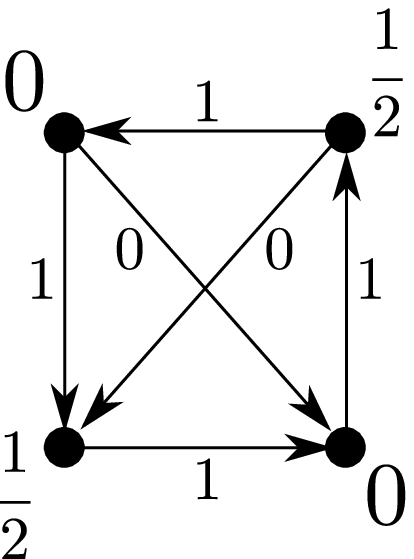}
\end{center}

Simiarly, Havet and Thomass\'{e}'s proof, using median orders, does not seem to easily generalize either.  Generalizing this to arc-weighted digraphs, a median order $v_1, \ldots, v_n$ minimizes the total arc-weight of arcs $v_i \to v_j$ where $j < i$.  For non-arc-weighted tournaments, they proved $v_n$ is a Seymour vertex.  However, that is not the case for arc-weighted tournaments, by taking Havet and Thomass\'{e}'s example for non-tournaments and adding arcs of weight zero.  

\begin{center}
\includegraphics[scale=0.8]{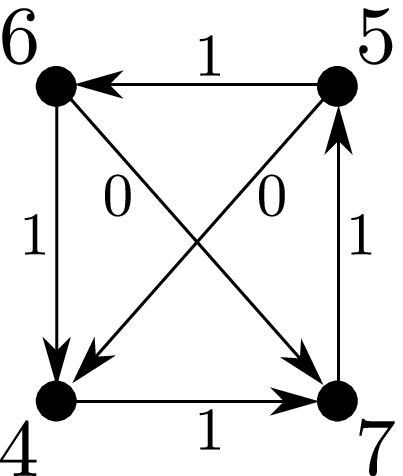}
\end{center}

This vertex-weighted digraph is meant to represent a digraph \emph{without} vertex-weights.  Instead, each vertex with a vertex-weight $k$ represents a tournament of $k$ vertices with $0$-weighted arcs between every pair of vertices, and an arc between two weighted vertices $u \to v$ represents a complete bipartite graph of edges from every vertex in the tournament $u$ to every vertex in the tournament $v$.    Here, some vertex in the tournament labeled $4$ will contain $v_n$, but $v_n$ will not be a Seymour vertex.

However, by extending the work of Fisher, we can generalize his proof to arc-weighted tournaments.   It depends on a nice property of arc-weighted diagraphs, where in some situations, one can contract along an arc and maintain the property of vertices being strongly contracting.

\begin{lemma}
\label{lemma:contract}
Let $D$ be an arc-weighted digraph, and let $u$ and $v$ be two vertices such that, for any vertex $x$, if $x \to u$ is an arc of nonzero weight, then $x \to v$ is also an arc (possibly weight zero).  Then by removing $u$ and adding all of the weight of $x \to u$ to $x \to v$ for all $x$, a new digraph $D'$ is created and $\delta^{D'}_y \leq \delta^D_y$ for all vertices $y \in V(D)$.  Such a maneuver is called a \emph{contraction} of $u$ to  $v$.
\end{lemma}

Notice that in this lemma statement, $x$ may equal $v$, and therefore the lemma does not apply if $v \to u$ is an arc of nonzero weight. We typically we will apply the lemma when $u \to v$ is an arc.    

\begin{proof}
Let the property P be the property that for any $x$ such that $x \to u$ is an arc of nonzero weight, we have $x \to v$ is also an arc (possibly of weight zero).

Let $y$ be any vertex.  We will show that $\alpha^{D'}_y = \alpha^D_y$ and $\beta^{D'}_y \leq \beta^D_y$, which will prove the lemma.  

We first show that $\alpha^{D'}_y = \alpha^D_y$.  If $y \to u$ is an arc of nonzero weight, then by property P, we have that $y \to v$ is an arc, and the only change in $\alpha^D_y$ is the weight transfer from $y \to u$ to $y \to v$ and the removal of $u$.  In this case, $\alpha^{D'}_y = \alpha^D_y$.  If $y \to u$ is not an arc or has zero weight, then the first neighborhood of $y$ either remains unchanged or loses an arc of weight zero, and hence $\alpha^{D'}_y = \alpha^D_y$.

We next show that $\beta^{D}_y \geq \beta^{D'}_y$.  Going from $D$ to $D'$, $y$ potentially loses second neighborhood weight on $u$ since $u$ is deleted, and potentially gains second neighborhood weight on $v$ because arcs into $v$ may gain weight.  Let $\gamma$ be the sum of all the other unaffected second neighborhood weight.  In other words,  $\gamma = \sum_{z \neq u, v} \beta_y(z)$.   Let $x$ be the vertex such that $y \to x$ is an arc and the second neighborhood weight of $v$ in $D'$ is $\max(w_{D'}(x v) - w_{D'}(y v), 0)$.   The second neighborhood weight of $u$ and $v$ in $D$ is at least $\max(w_D(x v) - w_D(y v), 0) + \max(w_D(x u) - w_D(y u, 0)$.  Thus, we have
\begin{eqnarray*}
 \beta^{D}_y & \geq & \max(w_D(x v) - w_D(y v), 0) + \max(w_D(x u) - w_D(y u), 0) + \gamma \\
& \geq & \max((w_D(x v) - w_D(y v)) + (w_D(x u) - w_D(y u)), 0) + \gamma \\
& = & \max((w_D(x v) + w_D(x u))  - (w_{D}(y v) + w_D(y u)), 0) + \gamma \\
& = & \max(w_{D'}(x v) - w_{D'}(y v), 0) + \gamma \\
& = &  \beta^{D'}_y.
\end{eqnarray*}
as desired.

\end{proof}

 Now the main result.

\begin{thm}
\label{thm:arc-tournaments}
Every arc-weighted tournament $D$ contains a Seymour vertex.  
\end{thm}

\begin{proof}
Suppose $D$ with arc-weights $w$ is a counterexample, meaning every vertex is strongly contracting.  Suppose further that $D$ and $w$ are chosen to minimize the number of vertices, and subject to this constraint, to maximize the number of arcs of weight zero.  There are three cases.

\textbf{There is a vertex-weighting $\eta$ of $\overleftarrow{D}$ such that every vertex of $\overleftarrow{D}$ is strongly contracting.}  By Proposition~\ref{prop:triangle-full}, there must be an arc $u \to v$ in $\overleftarrow{D}$ not contained in any directed triangles.  Thus, for any vertex $x$, if $x \to u \in A(\overleftarrow{D})$, then $v \to x$ cannot be an arc, since that would be a directed triangle.  Since this is a tournament, $x \to v \in A(\overleftarrow{D})$.  Now apply arc-weights $w'$ to $\overleftarrow{D}$  where $w'(uv) = \eta(v)$. We can then apply Lemma~\ref{lemma:contract} to $\overleftarrow{D}$ with arc-weights $w'$ to obtain a counterexample with fewer vertices than $D$, contradicting that $D$ had the minimum number of vertices.

Hence, we can assume $\overleftarrow{D}$ does not have a vertex-weighting such that every vertex is strongly contracting.  Therefore, by Theorem~\ref{thm:expanding-or-contracting}, $D$ has a vertex-weighting $\eta^*$ such that every vertex is weakly expanding.

\textbf{For every vertex $v$ and every second neighbor $x$, there is an arc of nonzero weight from $N_1^+(v)$ to $x$.}  Create new arc-weights $w^*$ where $w^*(uv) = w(uv) - \epsilon \eta^*(v)$ if $w(uv) > 0$, and $w^*(uv) = 0$ if $w(uv) = 0$.  Here, $\epsilon$ is chosen so that no arc changes to negative weight and at least one arc changes to weight zero.  

Notice that $D$ with arc-weights $w^*$ is still a counterexample.  For any vertex $v$, its second neighborhood weight is decreased by at least $\epsilon \eta^*(N_2^+(v))$, which doesn't count any lowering of second neighborhood weight within $N_1^+(v)$. The first neighborhood weight is lowered by $\epsilon \eta^*(N_1^+(v))$.  By the definition of $\eta^*$, we have $\eta^*(N_1^+(v)) \leq \eta^*(N_2^+(v))$, and hence the second neighborhood weight of every vertex decreased by at least as much as the first neighborhood weight.  Since $D$ with arc-weights $w^*$ is a counterexample with more arcs of weight zero than $w$, we have contradicted the fact that $w$ was chosen to maximize the number of arcs of weight $0$.  

\textbf{There exists a vertex $v$ and a second neighbor $x$ of $v$ such that all the arcs from $N_1^+(v)$ to $x$ are weight zero.}  If $y \to x$ is an arc of nonzero weight, we cannot have $v \to y$ as an arc, since then $y \to x$ would be an arc of nonzero weight from $N_1^+(v)$ to $x$.  Since $D$ is a tournament, $y \to v$ is an arc.  Then by Lemma~\ref{lemma:contract}, we can contract along $x \to v$ to obtain a counterexample with fewer vertices, contradicting the minimality of $D$.

\end{proof}

We remark that the proof of Theorem~\ref{thm:arc-tournaments} relies on the fact that for a tournament, either Proposition~\ref{prop:triangle-full} applies or Lemma~\ref{lemma:contract} applies.  This is not true for general digraphs, and it seems difficult to extend this proof method beyond tournaments.

\subsection*{Acknowledgements}  The author would like to thank Stephen Hartke for introducing him Fisher's work on this problem and Debbie Seacrest for her feedback and discussions.

\bibliographystyle{plain}
\bibliography{2ndNeighborhood}

\end{document}